\documentclass[12pt,twoside,a4paper,reqno]{amsart}
\usepackage{tikz}
\usetikzlibrary{matrix,arrows}
\usepackage{graphics}
\usepackage{amssymb}
\usepackage{eucal}
\usepackage{amsmath}

\newtheorem{thm}{Theorem}[section]

\newtheorem{pro}[thm]{Proposition}
\newtheorem{cor}[thm]{Corollary}
\theoremstyle{definition}
\newtheorem{den}[thm]{Definition}
\newtheorem{example}[thm]{Example}
\theoremstyle{remark}

\numberwithin{equation}{section}

\title{Connes-biprojective  dual Banach algebras}

\author{A. Shirinkalam$^1$}
\address{$^1$  Faculty of Mathematics and Computer
	Science, Amirkabir University of Technology, 424 Hafez Avenue,
	Tehran 15914, Iran, e-mail: {\tt shirinkalam\_a@aut.ac.ir}}

\author{A. Pourabbas$^2$}
\address{$^2$ Faculty of Mathematics and Computer
	Science, Amirkabir University of Technology, 424 Hafez Avenue,
	Tehran 15914, Iran, e-mail: {\tt arpabbas@aut.ac.ir}}

\begin{document}
\pagestyle{headings}

\begin{abstract}
In this paper, we introduce a new notion of  biprojectivity, called Connes-biprojective,
  for   dual Banach algebras.  We study the relation between this new notion to Connes-amenability and we  show that, 
for a given dual Banach algebra  $ \mathcal{A}  $, it is  Connes-amenable if and only if $ \mathcal{A}  $ is Connes-biprojective and has a bounded approximate identity.

Also,  for an Arens regular Banach algebra $ \mathcal{A}  $, we show that if $ \mathcal{A}  $ is biprojective, then the dual Banach algebra $ \mathcal{A} ^{**} $ is Connes-biprojective.
\end{abstract}
\maketitle

\section{Introduction and Preliminaries}

Let $ \mathcal{A}  $ be a Banach algebra and let $ \mathcal{A}\hat{\otimes} \mathcal{A}$ be the projective tensor product of $ \mathcal{A}  $ with
itself, which is a Banach $ \mathcal{A}  $-bimodule for the usual left and right operations with respect to $ \mathcal{A}  $. Let $ \pi :\mathcal{A}\hat{\otimes} \mathcal{A} \rightarrow  \mathcal{A}$ denote the  algebra multiplication in
$  \mathcal{A}$, defined by $ \pi (a\otimes b)=ab \quad (a, b \in \mathcal{A} ) $ and extended by linearity and continuity. Let $ \pi^* :  \mathcal{A}^* \rightarrow (\mathcal{A}\hat{\otimes} \mathcal{A})^*  $ be the adjoint map of $  \pi$,
where $ \mathcal{A}^* $ is the topological dual of $ \mathcal{A}  $.

In the \textit{Helemskii's Banach homology}  setting, there are two important notions related
to that one of \textit{(Johnson's) amenability}  for Banach algebras. These are \textit{projectivity}  and
\textit{flatness}. A Banach algebra $ \mathcal{A} $ is called biprojective if $  \pi$ is a retraction, which is to say, there
is a bounded $ \mathcal{A}  $-bimodule homomorphism  $ \rho :\mathcal{A} \rightarrow  \mathcal{A}\hat{\otimes} \mathcal{A} $ such that $\pi \circ \rho = id_{\mathcal{A}}  $. Similarly, A Banach algebra $ \mathcal{A} $ is called biflat if $ \pi^* $ is a coretraction; that is, if there is a bounded
$ \mathcal{A} $-bimodule homomorphism $ \gamma : (\mathcal{A}\hat{\otimes} \mathcal{A})^* \rightarrow \mathcal{A}^*$
such that 
$\gamma \circ \pi^* = id_{\mathcal{A}^*}  $ \cite [section VII]{H}. Then we have that
a Banach algebra $ \mathcal{A} $ is amenable if and only if $ \mathcal{A} $ is biflat and has a bounded approximate identity; see \cite [Proposition 4.3.23 and Exercise 4.3.15]{R} or \cite [Theorem VII.2.20]{H}. Clearly,
every biprojective Banach algebra is biflat -the converse is generally not true- and, as a
consequence, every biprojective Banach algebra with a bounded approximate identity is
amenable.

Another important characterization of amenability involving the map $ \pi :\mathcal{A}\hat{\otimes} \mathcal{A} \rightarrow  \mathcal{A}$ is
that $  \mathcal{A}$ is amenable if and only if there is a virtual diagonal for $  \mathcal{A}$. 

The class of dual Banach algebras was introduced by V. Runde (2001).
Examples of these algebras include von Neumann
algebras, the measure algebras $M(G)$ of a locally compact groups G, the algebra of bounded
operators $\mathcal{B}(H)$, for a Hilbert space or  a reflexive Banach space $H$, the bidual Banach algebra $  \mathcal{A}^{**} $ for  an Arens regular
algebras $   \mathcal{A}$. 

The original Johnson's amenability is too strong for dual
Banach algebras. Indeed, S. Wasserman proved (1976) that every amenable von
Neumann algebra is subhomogeneous, and H. G. Dales, F. Ghahramani and A. Helemskii
proved (2002) that $M(G)$ is amenable if and only if $G$ is amenable -as a group- and \textit{discrete}.

A suitable amenability-type concept to deal with dual Banach algebras is the \textit{Connes-amenability}, first introduced by Johnson, Kadison and
Ringrose for von Neumann algebras in 1972. A von Neumann algebra is Connes-amenable
if and only if it is injective. Also, $M(G)$ is Connes-amenable if and only if $G$ is
amenable 
(see  \cite {BP}, \cite {Co1} and \cite {Co2}).

It is possible to characterize Connes-amenability for dual Banach algebras in terms
of virtual diagonals. Let $   \mathcal{A}$ be a dual Banach algebra. For a given $   \mathcal{A}$-bimodule $E$,
let $\sigma WC(E)  $ denote the closed submodule of $E$ of all elements $x$ such that the mappings
$ \mathcal{A}\rightarrow E;\, a\mapsto a\cdot x $ and $ a\mapsto x\cdot a $ are $ \sigma(\mathcal{A},\mathcal{A}_*)-\sigma (E,E^*) $-continuous. Then $\mathcal{A}_* \subseteq  \sigma WC((\mathcal{A}\hat{\otimes} \mathcal{A})^*) $, from which it follows that $ \pi^* $ maps $ \mathcal{A}_* $ into $ \sigma WC((\mathcal{A}\hat{\otimes} \mathcal{A})^*)  $. Hence, $ \pi^{**} $ drops to an 
$ \mathcal{A} $-bimodule homomorphism $ \pi_{\sigma WC} : \sigma WC((\mathcal{A}\hat{\otimes} \mathcal{A})^*)^* \rightarrow \mathcal{A}$.
Any element $ M  $ in $ \sigma WC((\mathcal{A}\hat{\otimes} \mathcal{A})^*)^* $ satisfying 
$$ a \cdot M = M \cdot a  \quad\text{and} \quad a \cdot \pi_{\sigma WC} M=a \quad (a \in \mathcal{A}),$$
is called a $ \sigma WC $-virtual diagonal for $  \mathcal{A}$. Then, a dual Banach algebra $  \mathcal{A}$ is Connes-amenable
if and only if there exists a $ \sigma WC $-virtual diagonal for $  \mathcal{A}$  (Runde, 2004).

 It is  natural to look for
a suitable (Helemskii's) homological analogue of biprojectivity, in the setting
of dual Banach algebras, that fit well with Connes-amenability.

 In this paper, we introduce an analogue of biprojectivity, called Connes-biprojective, and we study its homological properties and we show that how this concept deals with Connes-amenability.

The organisation of this paper is as follows.
We introduce a suitable notion of Connes-biprojectivity in the setting of dual
Banach algebras. We show that a dual Banach algebra is Connes-amenable if and only if
it is Connes-biprojective and has a bounded approximate identity (identity, indeed).
We  prove that if an Arens regular Banach algebra $  \mathcal{A}$ is biprojective, then the
bidual Banach algebra $  \mathcal{A}^{**}$ is Connes-biprojective.

Given a Banach algebra $ \mathcal{A} $ and a Banach $ \mathcal{A} $-bimodule $ E $, the topological dual space $ E^* $ of $  E$
is a Banach $ \mathcal{A} $-bimodule with actions
$$ \langle x, a\cdot \varphi \rangle :=\langle x\cdot a, \varphi\rangle, \quad \langle x, \varphi\cdot a \rangle :=\langle a\cdot x, \varphi\rangle \quad  (a\in \mathcal{A}, x\in E, \varphi \in E^*).$$

\begin{den}\cite [Definition 1.1]{Ru1} 
	A Banach algebra $ \mathcal{A} $ is called dual, if it is a dual Banach space with predual $ \mathcal{A}_* $ such that the  multiplication in $ \mathcal{A} $ is separately $ \sigma (\mathcal{A}, \mathcal{A}_*) $-continuous. Equivalently, a Banach algebra $ \mathcal{A} $ is dual if it has a predual  $ \mathcal{A}_* $ which is a closed submodule of $ \mathcal{A}^* $ (\cite [Exercise 4.4.1]{R}).
\end{den}
Although a predual may not be unique, we can recognize it from the context. In particular, we may speak of the weak* topology on $  \mathcal{A} $ without ambiguity.

\begin{den}\cite [Definition 1.4]{R5}
	Let $ \mathcal{A} $ be a dual Banach algebra and let $ E $ be a dual  Banach $ \mathcal{A} $-bimodule. An element $ x\in E $ is called normal, if the maps $ a\mapsto a\cdot x $ and $ a\mapsto x\cdot a $ are weak* continuous.
\end{den}
The set of all normal elements in $ E $ is denoted by $ E_\sigma $. We say that $ E $ is normal if  $E=E_\sigma $. It is easy to see that $ E_\sigma $ is a norm-closed submodule of $ E $. However, there is no need  for $ E_\sigma $ to be weak*-closed.

For a given dual Banach algebra $ \mathcal{A} $  and a Banach $ \mathcal{A} $-bimodule $ E $,
it is easy to see that $ \sigma WC(E) $ is a closed $ \mathcal{A} $-submodule of $ E $ and so $ E $ is canonically mapped into $ \sigma WC (E^*)^* $. If $ F $ is another Banach $ \mathcal{A} $-bimodule and if $ \psi : E\rightarrow F $ is a bounded $ \mathcal{A} $-bimodule homomorphism, then $ \psi (\sigma WC(E))\subseteq \sigma WC(F) $ holds. Runde in \cite [Proposition  4.4]{R5} showed that $ E= \sigma WC(E)$ if and only if $ E^* $ is normal Banach $ \mathcal{A} $-bimodule, and therefore for any Banach $ \mathcal{A} $-bimodule $ E $, the dual module $ \sigma WC(E^*) $ is normal. 

Recall that for a Banach algebra $ \mathcal{A} $ and a Banach $ \mathcal{A}$-bimodule $ E $, a bounded linear map $  D:\mathcal{A}\rightarrow E $ is called a derivation if $ D(ab)=a\cdot D(b)+D(a)\cdot b $ for every $ a,b \in \mathcal{A} $. A derivation $  D:\mathcal{A}\rightarrow E $ is called inner if there exists  an element $ x \in E $ such that for every $ a\in \mathcal{A} $, $D(a)= a \cdot x - x \cdot a $.
\begin{den}\cite [Definition 1.5]{R5}
	A dual Banach algebra is called Connes-amenable if for every normal dual $ \mathcal{A} $-bimodule $ E $, every weak* continuous derivation $ D:\mathcal{A}\rightarrow E $ is inner.
\end{den}

\section{Connes-biprojective dual Banach algebras}
As mentioned in the introduction, the concept of biprojectivity is important  in the Helemskii's Banach homology and is closely related to Johnson's amenability. In this section we define a suitable analogue of Helemskii's homological-type concept to deal with dual Banach algebras, called Connes-biprojectivity and we show that this concept is closely related to Connes-amenability.

\begin{den}
	Let $  \mathcal{A} $ be a dual Banach algebra. Then $  \mathcal{A} $ is called \textit{Connes-biprojective}
	if there exists a bounded $  \mathcal{A} $-bimodule homomorphism $ \rho : \mathcal{A}\rightarrow \sigma WC(( \mathcal{A}\hat{\otimes}  \mathcal{A})^*)^* $ such that $  \pi_{\sigma WC}\circ \rho =id_{ \mathcal{A}} $ (that is, $ \pi_{\sigma WC} $ is a retraction).
\end{den}

In the following theorem, we determine the relation between Connes-biprojectivity and  Connes-amenability. 
\begin{thm}\label{3.5}
	The following are equivalent for a dual Banach algebra $  \mathcal{A} $:
	\begin{enumerate}
		\item[(i)] $  \mathcal{A} $ is Connes-biprojective and has a bounded approximate identity,
		\item[(ii)]$  \mathcal{A} $  is Connes-amenable.
	\end{enumerate}
\end{thm}
\begin{proof}
	Suppose that $  \mathcal{A} $ is Connes-amenable. Then there exists a $ \sigma WC $-virtual diagonal $ M\in  \sigma WC( (\mathcal{A}\hat{\otimes}  \mathcal{A})^*)^*$
	for $  \mathcal{A} $.
	We define $ \rho : \mathcal{A}\rightarrow (\sigma WC( \mathcal{A}\hat{\otimes}  \mathcal{A})^*)^* $ by $  \rho (a) = a\cdot M  $, for every $ a\in  \mathcal{A}. $ Then $$\| \rho (a) \| = \| a\cdot M \| \leq K \| a \| \| M \|.  $$
	Thus $ \rho $ is bounded. Also
	$  a\cdot \rho (b) = a\cdot (b\cdot M) = (ab)\cdot M$. On the other hand, since $ M $ is a virtual diagonal, $ (ab)\cdot M= M \cdot (ab)=(a\cdot M)\cdot b=\rho(a) \cdot b  ,$ that is, $ \rho $ is an $  \mathcal{A} $-bimodule homomorphism. It is easy to see that $  \pi_{\sigma WC}\circ \rho =id_{ \mathcal{A}} $. Note that since $  \mathcal{A} $ is Connes-amenable, it has an identity, equivalently, a bounded approximate identity.
	
	Conversely, suppose that  $  \mathcal{A} $ is Connes-biprojective and has a bounded approximate identity $ (e_\alpha)_\alpha $. Then there exists a bounded $  \mathcal{A} $-bimodule homomorphism $ \rho : \mathcal{A}\rightarrow \sigma WC(( \mathcal{A}\hat{\otimes}  \mathcal{A})^*)^* $ such that $  \pi_{\sigma WC}\circ \rho =id_{ \mathcal{A}} $. Let $  M_\alpha=\rho ( e_\alpha) \in \sigma WC(( \mathcal{A}\hat{\otimes}  \mathcal{A})^*)^* $. Then for every $ a \in  \mathcal{A}$, we have $$ a\cdot M_\alpha - M_\alpha \cdot a = \rho (a \cdot e_\alpha )- \rho (e_\alpha \cdot a) \rightarrow 0, $$ and $$ \Vert \pi_{\sigma WC} (M_ \alpha) (a)-a \Vert = \Vert  e_\alpha a-a \Vert \rightarrow 0\quad (\alpha \rightarrow \infty). $$
	Since  $ ( M_\alpha)_\alpha $ is uniformly bounded net in $ \sigma WC(( \mathcal{A}\hat{\otimes}  \mathcal{A})^*)^* $, it has a weak* limit point, say $ M $. It is easy to see that $ M $   is a $ \sigma WC $-virtual diagonal for $  \mathcal{A} $, so that $  \mathcal{A} $  is Connes-amenable.
\end{proof}

\begin{example}\label{3.3}
	
	Let $ \mathcal{A} $ be a biflat dual Banach algebra. Then  $ \pi ^* :\mathcal{A}^* \rightarrow(\mathcal{A} \hat{\otimes} \mathcal{A})^* $ is a co-retraction and so $   \pi ^*\vert _{\mathcal{A}_*} :\mathcal{A}_* \rightarrow \sigma WC((\mathcal{A}\hat{\otimes} \mathcal{A})^*)$ is a co-retraction again. This means that $  \pi_{\sigma WC} : (\sigma WC(\mathcal{A}\hat{\otimes} \mathcal{A})^*)^* \rightarrow \mathcal{A} $  is a retraction, that is, $ \mathcal{A} $ is  Connes-biprojective.
\end{example}
By the previous example, every biprojective dual Banach algebra is Connes-biprojective.
In the following example we see that the converse  is false in general.
\begin{example}\label{ex3.4}
	Let $ G $ be a non-discrete amenable  locally compact group. Then by \cite [Theorem 4.4.13]{R},
	$ M(G) $, the  measure algebra of $ G $, is a Connes-amenable dual Banach algebra and thus by
	Theorem \ref{3.5}, $ M(G) $ is Connes-biprojective.  Since $ G $ is not discrete, by \cite [Theorem 1.3]{DGH} $ M(G) $ is not amenable, so it is not biflat.
\end{example}
Clearly, every Connes-amenable Banach algebra is Connes-biprojective.
Here we give  two examples of  Connes-biprojective dual Banach algebras, which are not Connes-amenable.
\begin{example}
	Let $S$ be a discrete semigroup and  let $\ell^{1}(S)$ be  its semigroup algebra. Let $ \mathcal{A}=\ell^{1}(S)^{*} $. If $\phi$ is a character on $c_{0}(S)$,  then there exists a unique extension of $\phi$ on $c_{0}(S)^{**}$, (which is denoted    by $\tilde{\phi}$) and defined by $$\tilde{\phi}(F)=F(\phi) \quad (F\in c_{0}(S)^{**}\cong  \mathcal{A}).$$
	$ \tilde{\phi} $ is a multiplicative map because, for every $ F,G \in \mathcal{A} $,
	$$ \tilde{\phi}(FG)=(FG)(\phi )=F(\phi) G(\phi).$$
	Now we define  a new multiplication on $ \mathcal{A}$ by
	$$ab=\tilde{\phi}(a)b\quad (a,b\in  \mathcal{A},\quad \phi\in c_{0}(S)^{*}).$$
	With this multiplication, $ \mathcal{A}$ becomes a Banach algebra which is a dual Banach space. We denote this algebra by $ \mathcal{A}_{\phi}$. We define a map
	$\rho: \mathcal{A}_{\phi}\rightarrow  \mathcal{A}_{\phi} \hat{\otimes} \mathcal{A}_{\phi}$  by $\rho(a)=x_{0}\otimes a$, where $a\in  \mathcal{A}_{\phi}$ and $x_{0}$ is    an element in $ \mathcal{A}_{\phi}$  such that $\tilde{\phi}(x_{0})=1.$
	It is easy to see that $\rho$ is a bounded $ \mathcal{A}_{\phi}$-bimodule homomorphism and $\pi_{ \mathcal{A}_{\phi}}\circ\rho(a)=a$ for each $a\in  \mathcal{A}_{\phi}.$
	Hence $ \mathcal{A}_{\phi}$ is a biprojective Banach algebra.  If we show that $ \mathcal{A}_{\phi}$ is a dual Banach algebra, then by  Example \ref{3.3}  $ \mathcal{A}_{\phi}$ is Connes-biprojective.
	It is enough to show that the multiplication in $ \mathcal{A}_{\phi}$ is separately $w^{*}$-continuous. Suppose  $b\in  \mathcal{A}$ is a fixed element. Let $ a \in  \mathcal{A} $ and $(a_{\alpha}) \subseteq  \mathcal{A}$   be such  that $a_{\alpha} \rightarrow a$ in the $w^*$-topology. For each $f\in \ell^{1}(S)\cong c_{0}(S)^{*}$, we have
	$$a_{\alpha}(f)\rightarrow a(f) \quad \text{as} \;\alpha \rightarrow \infty.$$
	Since the character space  of $ c_{0}(S)$ lies in  $\ell^{1}(S)$, so the last statement is also true for $f=\phi$, where $\phi$ is the corresponding functional to  $\tilde{\phi}$. Hence, $a_{\alpha}(\phi)\rightarrow a(\phi)\quad \text{as} \;\alpha \rightarrow \infty$, which implies that $\tilde{\phi}(a_{\alpha})\rightarrow \tilde{\phi}(a) \quad \text{as} \;\alpha \rightarrow \infty$. Then $\tilde{\phi}(a_{\alpha})b\rightarrow \tilde{\phi}(a)b$. Now suppose that $ (b_{\alpha}) $  is a net in $ \mathcal{A}_{\phi}$ such that $b_{\alpha}\rightarrow b \quad \text{as} \;\alpha \rightarrow \infty$ in the $w^*$-topology. It is easy to see that $\tilde{\phi}(a)b_{\alpha}\rightarrow\tilde{\phi}(a)b$ in the $w^*$-topology, for every   $a\in  \mathcal{A}_{\phi}$. Hence $ \mathcal{A}_{\phi}$ is a dual Banach algebra. \\
	Now if $ \mathcal{A}_{\phi}$ is  Connes-amenable, then $ \mathcal{A}_{\phi}$ has an  identity, so  $\dim  \mathcal{A}=1$, (because for each  $a\in  \mathcal{A}$  we have  $a=ae=\tilde{\phi}(a)e$, where $e$ is an  identity of $ \mathcal{A}_{\phi}$), hence a
	contradiction reveals.
\end{example}
\begin{example}\label{exa3.2}
	Consider $  \mathcal{A}= \begin{pmatrix}
	0 & \mathbb{C} \\
	0 & \mathbb{C}
	\end{pmatrix}$
	with usual matrix multiplication and $ L^1 $-norm. Since $  \mathcal{A} $ is finite dimensional, it is a dual Banach algebra. Clearly $  \mathcal{A} $ has a right identity but it does not have an identity,  so  it  is not Connes-amenable. We define a map $ \zeta : \mathcal{A} \rightarrow  \mathcal{A}\hat{\otimes}  \mathcal{A}$ by\\
	$ \begin{pmatrix}
	0 & x \\
	0 & y
	\end{pmatrix} \mapsto \begin{pmatrix}
	0 & x \\
	0 & y
	\end{pmatrix} \otimes \begin{pmatrix}
	0 & 1 \\
	0 & 1
	\end{pmatrix}$.
	We have
	$ \| \zeta \begin{pmatrix}
	0 & x \\
	0 & y
	\end{pmatrix} \| = \| \begin{pmatrix}
	0 & x \\
	0 & y
	\end{pmatrix} \otimes \begin{pmatrix}
	0 & 1 \\
	0 & 1
	\end{pmatrix} \| = \|  \begin{pmatrix}
	0 & x \\
	0 & y
	\end{pmatrix} \|  \|   \begin{pmatrix}
	0 & 1 \\
	0 & 1
	\end{pmatrix}  \| = 2 \|  \begin{pmatrix}
	0 & x \\
	0 & y
	\end{pmatrix}    \| $
	so $ \|\zeta \| \leq 2 $ and an easy calculation shows that $\zeta $ is an $  \mathcal{A} $-bimodule homomorphism. By composing the canonical map $  \mathcal{A}\hat{\otimes}  \mathcal{A}\rightarrow  \sigma WC(( \mathcal{A}\hat{\otimes}  \mathcal{A})^*)^* $ with $ \zeta $, we obtain a bounded $  \mathcal{A} $-bimodule homomorphism $ \rho : \mathcal{A} \rightarrow \sigma WC(( \mathcal{A}\hat{\otimes}  \mathcal{A})^*)^* $. Since for every $ a \in  \mathcal{A},  \pi_{\sigma WC} \circ \rho (a)= \pi_{\sigma WC} \circ \zeta (a) $ we have,\\
	$ (\pi_{\sigma WC} \circ \rho ) \begin{pmatrix}
	0 & x \\
	0 & y
	\end{pmatrix} = \pi_{\sigma WC}( \begin{pmatrix}
	0 & x \\
	0 & y
	\end{pmatrix} \otimes \begin{pmatrix}
	0 & 1 \\
	0 & 1
	\end{pmatrix} )= \begin{pmatrix}
	0 & x \\
	0 & y
	\end{pmatrix}.$
	This means that $ \rho $ is a right inverse of $  \pi_{\sigma WC} $. Hence $  \mathcal{A} $ is Connes-biprojective.
\end{example}

The next theorem shows that how Connes-biprojectivity deals with homomorphisms.

\begin{thm}\label{hard}
	Let $  \mathcal{A} $ be a    Banach algebra, and  let $ \mathcal{B} $ be a dual  Banach algebra.  Let $\theta :  \mathcal{A} \rightarrow \mathcal{B} $ be a continuous   homomorphism.
	\begin{enumerate}
		\item[(i)] Let $ \theta^* :\mathcal{B}^*\rightarrow \mathcal{A}^*  $ be such that $\theta^*\vert_{\mathcal{B}_*} :\mathcal{B}_*\rightarrow \mathcal{A}^*   $ is surjective.   Suppose that the image of the closed unit ball of $  \mathcal{A} $ is weak* dense in the closed unit ball of  $  \mathcal{B}  $. Then, biprojectivity of  $  \mathcal{A} $ implies Connes-biprojectivity of $ \mathcal{B} $.
		
		\item[(ii)]If $  \mathcal{A} $  is dual and  Connes-biprojective and  $\theta $ is weak* continuous, then $ \mathcal{B} $  is   Connes-biprojective.
	\end{enumerate}
\end{thm}
\begin{proof}
	Note that since $\theta :  \mathcal{A} \rightarrow \mathcal{B} $ is a continuous homomorphism, by \cite [Proposition 1.10.10]{Pal}, the map $ \theta \otimes \theta : \mathcal{A} \otimes \mathcal{A} \rightarrow \mathcal{B} \otimes \mathcal{B} $ defined by
	$  (\theta \otimes \theta) (a \otimes b) = \theta (a) \otimes \theta (b) $, can be extended
	to a bounded linear map  $ \theta \hat{\otimes} \theta :\mathcal{A} \hat{ \otimes} \mathcal{A} \rightarrow \mathcal{B} \hat{ \otimes} \mathcal{B}. $  It is readily seen that,
	for every $ a, b, c \in \mathcal{A} $, $ (\theta \hat{\otimes} \theta)([a\otimes b] \cdot c)=(\theta \hat{\otimes} \theta)(a \otimes b)\bullet \theta(c)   $,
	similarly $ (\theta \hat{\otimes} \theta)(c \cdot [a\otimes b]  )=\theta(c) \bullet  (\theta \hat{\otimes} \theta)(a \otimes b)   $ where ``$ \cdot $" is the action of $ \mathcal{A} $ on $ \mathcal{A} \hat{ \otimes} \mathcal{A} $ and ``$ \bullet $" is the action of $ \mathcal{B} $ on $ \mathcal{B} \hat{ \otimes} \mathcal{B} $ inherited by $ \theta $. 
	Also, consider  the canonical  map $ \imath : \mathcal{B} \hat{ \otimes} \mathcal{B} \rightarrow  \sigma WC((\mathcal{B}\hat{\otimes} \mathcal{B})^*)^*,$ which is norm continuous with weak* dense range.

	(i) Suppose that $ \mathcal{A}  $ is biprojective. Then there is an $\mathcal{A}  $-bimodule homomorphism $\rho^{\mathcal{A}} :\mathcal{A} \rightarrow \mathcal{A} \hat{ \otimes} \mathcal{A}  $ such that $ \pi^{\mathcal{A}} \circ \rho^{\mathcal{A}} = id_{\mathcal{A}}$.    Define a map $\zeta :\mathcal{A} \rightarrow \sigma WC((\mathcal{B}\hat{\otimes} \mathcal{B})^*)^*  $ given  by $ \zeta (a) = (\imath \circ (\theta \hat{\otimes} \theta) \circ \rho^{\mathcal{A}})(a)  $. Then $  \zeta$ is  bounded, and for every $ a, a' \in \mathcal{A} $  we have
	\begin{eqnarray}\label{1111}
	\zeta (aa')= \theta (a) \bullet \zeta (a')= \zeta(a) \bullet \theta (a').
	\end{eqnarray}
	
	For every $ b \in  \mathcal{B}$, by assumption, there exists a bounded net $ (a_\alpha) \subseteq  \mathcal{A}$ such that $ \theta  (a_\alpha)\rightarrow b $ in the weak* topology.
	Since $\zeta $ is bounded, $(\zeta (a_{\alpha}))_{\alpha}  $ has a weak* accumulation point   by  Banach-Alaoglu theorem. Passing to a subnet (if it is necessary),  $ w^*-\lim _{\alpha} \zeta (a_{\alpha})$ exists. Thus,
	we can extend $ \zeta $ to a weak* continuous map $ \rho ^\mathcal{B} : \mathcal{B} \rightarrow \sigma WC((\mathcal{B}\hat{\otimes} \mathcal{B})^*)^*$ defined by $  \rho ^\mathcal{B}(b)= w^*-\lim _{\alpha} \zeta (a_{\alpha})$.
	We need to verify that $ \rho ^\mathcal{B} $ is well-defined. For this, it is enough to show that $ w^*-\lim _{\alpha} \zeta (a_{\alpha})=0  $ in $ \sigma WC((\mathcal{B}\hat{\otimes} \mathcal{B})^*)^* $, whenever $w^*-\lim _{\alpha} \theta  (a_\alpha) =0  $ in $ \mathcal{B} $.
	If  $ \lambda \in \mathcal{A}^* $, then by assumption,  there is a $ \varphi \in  \mathcal{B}_*$ such that $\theta^*\vert_{\mathcal{B}_*}(\varphi)=\lambda  $. Now we have
	$$ \lim _\alpha \langle \lambda ,  a_\alpha \rangle = \lim _\alpha \langle \theta^*(\varphi) ,  a_\alpha \rangle = \lim _\alpha \langle \varphi , \theta (a_\alpha)\rangle =0.$$ Hence $a_\alpha \rightarrow 0  $  in the weak topology of $ \mathcal{A} $. Since $ \zeta $
	is weak-weak* continuous, we conclude that $\zeta (a_{\alpha}) \rightarrow 0   $ in the weak* topology.

	Suppose that $ b$ and $ b' \in \mathcal{B} $. Then there exist two nets $ (a_ \alpha )$ and $ (a'_\beta) $ in $ \mathcal{A} $ such that
	$ \theta ( a_ \alpha ) \rightarrow b$ and $ \theta (a'_\beta)\rightarrow b' $ in the weak* topology. By the equation (\ref{1111}), and by the weak* continuity of the action of $\mathcal{B}  $ we have,
	$\rho ^\mathcal{B}(bb')=w^*-\lim _{\alpha} w^*-\lim _{\beta} \zeta(  a_\alpha  a'_\beta )
	=w^*-\lim _{\alpha} w^*-\lim _{\beta}\theta ( a_\alpha ) \bullet \zeta(a'_\beta) =b\bullet \rho ^\mathcal{B}(b'). $
	Similarly, $ \rho ^\mathcal{B}(bb')=\rho ^\mathcal{B}(b)\bullet b'  $. Thus $ \rho ^\mathcal{B} $ is a $ \mathcal{B} $-bimodule homomorphism.
	
	Finally, we prove that $ \pi^{\mathcal{B}}_{\sigma WC}\circ \rho ^\mathcal{B}(b)=b $ for every $ b \in  \mathcal{B}$. Observe that for the elementary tensor element $ a\otimes a' \in \mathcal{A} \hat{\otimes} \mathcal{A} $ we have
	$$
	\pi^{\mathcal{B}}_{\sigma WC} \circ \imath \circ (\theta \hat{\otimes}  \theta ) (a\otimes a' ) = \pi^{\mathcal{B}}_{\sigma WC} \circ \imath (\theta (a) \otimes \theta (a'))
	=\theta (aa')
	=\theta \circ \pi ^\mathcal{A} (a\otimes a' ).
	$$
	Thus for every $ a \in \mathcal{A} $,
	\begin{eqnarray}\label{2.3}
	\pi^{\mathcal{B}}_{\sigma WC} \circ \zeta(a)=\theta (a).
	\end{eqnarray}
	Now let  $ b \in  \mathcal{B}$ and take a net $ (a_ \alpha ) \subseteq  \mathcal{A}$ such that $ \theta (a_ \alpha )  \rightarrow b $ in the weak*-topology. Then, by the   equation (\ref{2.3}), we have $ b=w^*-\lim _{\alpha} \theta (a_ \alpha ) =w^*-\lim _{\alpha}\pi^{\mathcal{B}}_{\sigma WC} \circ  \zeta (a_ \alpha) = \pi^{\mathcal{B}}_{\sigma WC}(w^*-\lim _{\alpha}  \zeta (a_ \alpha)) = \pi^{\mathcal{B}}_{\sigma WC} \circ \rho ^\mathcal{B}(b).$

	(ii) Suppose that $  \mathcal{A} $  is dual and  Connes-biprojective. Then there is  an $\mathcal{A}  $-bimodule homomorphism $\rho^{\mathcal{A}} :\mathcal{A} \rightarrow \sigma WC ((\mathcal{A} \hat{ \otimes} \mathcal{A})^* )^* $ such that $ \pi^{\mathcal{A}}_{\sigma WC} \circ \rho^{\mathcal{A}} = id_{\mathcal{A}}$.
	Consider the map  $ (\theta \hat{\otimes} \theta)^*:(\mathcal{B} \hat{ \otimes} \mathcal{B})^*\rightarrow (\mathcal{A} \hat{ \otimes} \mathcal{A})^* $
	which is  an $\mathcal{A}$-bimodule map. We conclude that  $ (\theta \hat{\otimes} \theta)^*$ maps
	$\sigma WC (\mathcal{B} \hat{ \otimes} \mathcal{B})^* $ into $\sigma WC (\mathcal{A} \hat{ \otimes} \mathcal{A})^*   $.
	 Consequently,  we obtain a weak* continuous map
	 $$ \varphi :=((\theta \hat{\otimes} \theta)^*|_{\sigma WC (\mathcal{B} \hat{ \otimes} \mathcal{B})^*})^* :\sigma WC ((\mathcal{A} \hat{ \otimes} \mathcal{A})^*)^* \rightarrow \sigma WC( (\mathcal{B} \hat{ \otimes} \mathcal{B})^*)^*. $$ Now, we define a map
	$ \zeta :\mathcal{A} \rightarrow \sigma WC ((\mathcal{B} \hat{ \otimes} \mathcal{B})^*)^* $ given by $ \zeta (a) = \varphi \circ \rho ^{\mathcal{A}}(a) $, for every $ a \in \mathcal{A} $. Since $ \varphi  $ and $  \rho ^{\mathcal{A}}$ are weak* continuous, so is $ \zeta $ and since $\theta $ is weak* continuous and  weak* dense range, $ \zeta $ extends to  a weak* continuous map $ \rho ^\mathcal{B} : \mathcal{B} \rightarrow  \sigma WC ((\mathcal{B} \hat{ \otimes} \mathcal{B})^*)^*$. An  argument similar to (i) shows that $ \rho ^\mathcal{B} $ is a $\mathcal{B}  $-bimodule homomorphism.
	
	Let  $ a\otimes a'\in \mathcal{A} \hat{ \otimes} \mathcal{A} $. Since the map $ \varphi $ is the double transpose of $ \theta \hat{\otimes} \theta $, we have
	$$ \varphi (a\otimes a') = (\theta \hat{\otimes} \theta) (a\otimes a')= (\theta (a) \otimes \theta (a')).  $$ Hence
	$$\pi_{\sigma WC}^{\mathcal{B}}\circ \varphi  (a\otimes a') =\pi_{\sigma WC}^{\mathcal{B}}  (\theta (a) \otimes \theta (a')) = \theta (aa') = \theta  \circ \pi_{\sigma WC}^{\mathcal{A}} (a\otimes a').$$
	Thus by linearity and continuity and by the hypothesis, for every $ a \in \mathcal{A} $,  we have
	\begin{eqnarray}\label{2}
	 \pi_{\sigma WC}^{\mathcal{B}}\circ \varphi \circ \rho^{\mathcal{A}}(a)= \theta \circ    \pi^{\mathcal{A}}_{\sigma WC} \circ \rho^{\mathcal{A}} (a) = \theta (a),
	\end{eqnarray}
	 so that the following diagram is commutative;
	\begin{center}
		\begin{tikzpicture}
		\matrix [matrix of math nodes,row sep=2cm,column sep=2cm,minimum width=2cm]
		{
			|(A)| \displaystyle \mathcal{A} &   |(B)| \sigma WC( (\mathcal{A} \hat{ \otimes} \mathcal{A})^*)^*       \\
			|(C)|   \mathcal{B}                     &   |(D)| \sigma WC( (\mathcal{B} \hat{ \otimes} \mathcal{B})^*)^*.  \\
		};
		\draw[->]  (A)-- node [above] { $  \rho^{\mathcal{A}} $}(B);
		\draw[->]   (A)--  node [left] { $ \theta $} (C);
		\draw[->]  (D)-- node [below]   { $ \pi_{\sigma WC}^{\mathcal{B}} $}(C);
		\draw[->]  (B)-- node [right]  {$ \varphi $}(D);
		\end{tikzpicture}
	\end{center}
	
	Since the range of $ \theta $ is dense, by (\ref{2}) we have
	$\pi_{\sigma WC}^{\mathcal{B}}\circ \rho ^\mathcal{B}(b)=b  $  for every $ b \in \mathcal{B} $.
	Hence $ \mathcal{B} $ is Connes-biprojective.
\end{proof}

Given a Banach algebra $\mathcal{A}$, we may define  bilinear maps $ \mathcal{A}^{**} \times \mathcal{A}^{*} \rightarrow \mathcal{A}^{*}$ and $ \mathcal{A}^{*} \times \mathcal{A}^{**} \rightarrow \mathcal{A}^{*}$  given by $(\Phi , \mu )\mapsto \Phi \cdot \mu $ and  $(\mu , \Phi)\mapsto \mu \cdot \Phi $, respectively, where
for every $\Phi \in \mathcal{A}^{**},\,\,\mu \in \mathcal{A}^{*}$ and  $a \in \mathcal{A}$
$$ \langle \Phi \cdot \mu, a \rangle= \langle \Phi,\mu \cdot a  \rangle, \qquad \langle   \mu \cdot \Phi , a \rangle =\langle  \Phi , a \cdot \mu \rangle .$$
We  define two bilinear maps $  \square , \lozenge : \mathcal{A}^{**} \times \mathcal{A}^{**} \rightarrow \mathcal{A}^{**}$ given by $ (\Phi ,  \Psi) \mapsto \Phi \square  \Psi $ and $ (\Phi ,  \Psi) \mapsto \Phi \lozenge \Psi $, where for every $\Phi, \Psi\in\mathcal{A}^{**}$ and $\mu\in \mathcal{A}^{*}$ 
$$ \langle \Phi \square  \Psi ,\mu \rangle= \langle \Phi,\Psi  \cdot \mu \rangle, \qquad \langle  \Phi \lozenge \Psi , \mu    \rangle =\langle  \Psi , \mu \cdot \Phi  \rangle, $$
One can check that $\square  $ and $  \lozenge$ are actually algebra products, called the first and the
second Arens products, respectively.
If for every $ \Phi , \Psi \in \mathcal{A}^{**} $  we have $\Phi  \square \Psi=\Phi  \lozenge \Psi$, then we say that $  \mathcal{A}$ is Arens
regular.

It is well-known that if $ \mathcal{A} $ is an Arens regular Banach algebra, then $ \mathcal{A}^{**} $, the bidual  of $ \mathcal{A} $, is a dual Banach algebra with predual $ \mathcal{A}^* $(see \cite{R5} for more details).
\begin{cor}
	Let $ \mathcal{A} $ be an Arens regular Banach algebra. If $ \mathcal{A} $ is biprojective, then $ \mathcal{A}^{**} $ is
	Connes-biprojective.
\end{cor}
\begin{proof}
	It is clear that the inclusion map  $id: \mathcal{A} \hookrightarrow \mathcal{A}^{**}$ satisfies the conditions of Theorem \ref{hard}(i).
\end{proof}
If $ \mathcal{A} $ is a dual Banach algebra and $ I $ is a weak*-closed ideal of $ \mathcal{A} $, then $ I $ is a dual Banach algebra with predual $ I_*=\mathcal{A}_*/I^\perp $. To see this, we have
$$ (I_*)^*=( \mathcal{A}_*/I^\perp)^*=(I^\perp)^\perp = I.$$
Since the multiplication in $ \mathcal{A} $ is separately weak* continuous, a simple verification shows that the multiplication on $ \mathcal{A}/I $ is separately weak* continuous, so $ \mathcal{A}/I  $ is also a dual Banach algebra.
\begin{pro}
	Let $ \mathcal{A} $ be a Connes-biprojective dual Banach algebra and  let $ I $ be a weak*-closed ideal of $ \mathcal{A} $ which is essential as a left Banach $ \mathcal{A} $-module. Then $  \mathcal{A}/I  $ is Connes-biprojective.
\end{pro}
\begin{proof}
	Since $ \mathcal{A} $ is  Connes-biprojective, the map\\
	 $ \pi_{\sigma WC} : \sigma WC((\mathcal{A}\hat{\otimes} \mathcal{A})^*)^* \rightarrow \mathcal{A}$ is a retraction, so there exists
	a bounded $ \mathcal{A} $-bimodule homomorphism $ \rho :\mathcal{A}\rightarrow \sigma WC((\mathcal{A}\hat{\otimes} \mathcal{A})^*)^* $ as a right inverse of $ \pi_{\sigma WC} $. Let $ q:\mathcal{A}  \rightarrow  \mathcal{A}/I  $ be the quotient map. Then the map $ id \hat{\otimes} q : \mathcal{A}\hat{\otimes} \mathcal{A} \rightarrow \mathcal{A}\hat{\otimes}  (\mathcal{A}/I ) $ is a bounded $ \mathcal{A} $-bimodule homomorphism, so is $ (id \hat{\otimes} q )^*: ( \mathcal{A}\hat{\otimes}  (\mathcal{A}/I  ) )^*\rightarrow  ( \mathcal{A}\hat{\otimes} \mathcal{A}  )^* $. Thus $ (id \hat{\otimes} q )^* $ maps $ \sigma WC((\mathcal{A}\hat{\otimes} (\mathcal{A}/I  ))^*) $ into $ \sigma WC((\mathcal{A}\hat{\otimes} \mathcal{A})^*) $,
	 therefore we obtain an $ \mathcal{A} $-bimodule homomorphism
	$$ ( (id \hat{\otimes} q )^*\vert_{\sigma WC(\mathcal{A}\hat{\otimes} \mathcal{A})^*})^*: \sigma WC((\mathcal{A}\hat{\otimes} \mathcal{A})^*)^* \rightarrow  \sigma WC((\mathcal{A}\hat{\otimes} (\mathcal{A}/I  ))^*)^* .$$
	Composing this map with $ \rho $, we obtain the map $$ \varphi :=  ( (id \hat{\otimes} q )^*\vert_{\sigma WC(\mathcal{A}\hat{\otimes} \mathcal{A})^*})^* \circ \rho :\mathcal{A} \rightarrow \sigma WC((\mathcal{A}\hat{\otimes} (\mathcal{A}/I  ))^*)^* ,$$ which is a bounded $ \mathcal{A} $-bimodule homomorphism since it is a composition of  such two maps.
	If $ a\in \mathcal{A} $, then $$ \varphi (a) \in \sigma WC((\mathcal{A}\hat{\otimes}(\mathcal{A}/I  ))^*)^*=
	(\mathcal{A}\hat{\otimes} (\mathcal{A}/I  ))^{**} /\sigma WC((\mathcal{A}\hat{\otimes}(\mathcal{A}/I  ))^*)^\perp. $$  Suppose that $ \tilde{\varphi} (a)$ is a corresponding element of $ \varphi (a) $ in $ (\mathcal{A}\hat{\otimes} (\mathcal{A}/I  ))^{**}  $ (that is, $ \tilde{\varphi} (a)$ is the image of $ \varphi (a) $ under the quotient map). By the Goldstine's theorem there is a net $ (\varphi _\alpha(a))_\alpha \subseteq  \mathcal{A}\hat{\otimes} (\mathcal{A}/I  )$ such that $\varphi _\alpha(a) \rightarrow  \tilde{\varphi} (a) $ in the weak* topology. Since $\sigma WC((\mathcal{A}\hat{\otimes} (\mathcal{A}/I  ))^*)^*$ is normal, we have $ \varphi _\alpha(a)\cdot i \rightarrow \tilde{\varphi} (a) \cdot i $ for every $ i\in I $, and since for every $ \alpha $  we have $\varphi _\alpha(a)\cdot i=0 $, hence $ \tilde{\varphi} (a) \cdot i=0.  $ Thus $ \varphi (ai)=\varphi (a)\cdot i=0 $. By the fact that $ I $ is a left essential ideal, we have $ \varphi \vert _I=0 $, so it induces an $ \mathcal{A} $-bimodule homomorphism
	$ \hat{\varphi}:(\mathcal{A}/I  )\rightarrow \sigma WC((\mathcal{A}\hat{\otimes} (\mathcal{A}/I  ))^*)^*. $

In contrast, we have a  bounded $ \mathcal{A} $-bimodule homomorphism $$  q  \hat{\otimes} id :\mathcal{A}\hat{\otimes} (\mathcal{A}/I  )\rightarrow (\mathcal{A}/I  )\hat{\otimes} (\mathcal{A}/I),$$ which gives an $ \mathcal{A} $-bimodule homomorphism $$  (q \hat{\otimes} id)^* :((\mathcal{A}/I  )\hat{\otimes} (\mathcal{A}/I  ))^* \rightarrow (\mathcal{A}\hat{\otimes} (\mathcal{A}/I  ))^*, $$ such that it maps $ \sigma WC( ((\mathcal{A}/I  )\hat{\otimes} (\mathcal{A}/I  ))^* ) $ into $ \sigma WC (  (\mathcal{A}\hat{\otimes} (\mathcal{A}/I  ))^*) $.
Now consider the adjoint map $(q \hat{\otimes} id)^*\vert_{\sigma WC (\frac{\mathcal{A}}{I}\hat{\otimes} \frac{\mathcal{A}}{I})^* })^*$ which is denoted by
	$$\psi :\sigma WC(( (\mathcal{A}\hat{\otimes} (\mathcal{A}/I  ))^*)^* \rightarrow \sigma WC( ((\mathcal{A}/I  )\hat{\otimes} (\mathcal{A}/I  ))^*)^*.$$
Take
	 $ \zeta :=\psi \circ \hat{\varphi}: \mathcal{A}/I  \rightarrow \sigma WC( ((\mathcal{A}/I  )\hat{\otimes} (\mathcal{A}/I  ))^*)^* $. It is easy to see that $ \pi_{\sigma WC(\mathcal{A}/I  )}\circ \zeta =id_{\mathcal{A}/I  } $, therefore $ \mathcal{A}/I  $ is Connes-biprojective.
\end{proof}

\bigskip



\begin{thebibliography}{99}
\setlength{\baselineskip}{.45cm}






\bibitem{BP}
{\it J. W. Bunce and W. L. Paschke},  Quasi-expectations and amenable von Neumann
algebras,   Proc. Amer. Math. Soc.  {\bf 71} (1978), 232-236.


\bibitem{Co1}
{\it A. Connes},  Classification of injective factors,   Ann.  Math.   {\bf 104} (1976),
73-114.
\bibitem{Co2}
{\it A. Connes},  On the cohomology of operator algebras,   J.  Func. Anal.  {\bf 28} (1978),
248-253.



\bibitem{DGH}
{\it H. G. Dales, F. G hahramani and A. Ya. Helemskii},   Amenability of measure algebras,   J. London Math. Soc.  {\bf 66} (2002), 213-226.


\bibitem{H}
{\it A. Ya. Helemskii},   The Homology of Banach and Topological Algebras, Kluwer
Acad. Pub. vol. 41, Dordrecht 1989.

\bibitem{Pal}
{\it T. W. Palmer},   Banach Algebras and the General Theory  of $ * $-Algebras, Vol. I, Cambridge University Press, 1994.

\bibitem{Ru1}
{\it V. Runde}, Amenability for dual Banach algebras,  Studia Math.  {\bf 148} (2001), 47-66.



\bibitem{R5}
---------,  Dual Banach algebras: Connes-amenability, normal, virtual diagonals, and injectivity of the predual bimodule,
  Math. Scand.  {\bf 95} (2004), 124-144.

\bibitem{R}
---------,   Lectures on Amenability, Lecture Notes in Mathematics, Vol. 1774, Springer-Verlag, Berlin, 2002.


 
 

\end{thebibliography}
\end{document}